\documentclass[11pt]{amsart}
\usepackage{amsmath}
\usepackage{mathtools}
\usepackage{esint}
\usepackage{tikz}
\usepackage{comment}
\usepackage{mathrsfs}

\theoremstyle{plain}
\newtheorem{theorem}{Theorem}[section]
\newtheorem{lemma}[theorem]{Lemma}

\theoremstyle{definition}
\newtheorem{remark}[theorem]{Remark}

\numberwithin{equation}{section}

\def\R{\mathbf R}
\def\N{\mathbf N}
\def\S{\Sigma}
\def\epsilon{\varepsilon}
\def\div{\mathrm{div}}
\def\({\left(}
\def\){\right)}

\begin{document}

\title[Liouville Type Theorems]{Liouville Type Theorems for Minimal Surface Equations in Half Space}

\author{Guosheng Jiang}
\address{Beijing International Center for Mathematical Research, Peking University, Beijing 100871, China}
\email{gsjiang@pku.edu.cn}
\author{Zhehui Wang}
\address{Beijing International Center for Mathematical Research, Peking University, Beijing 100871, China}
\email{wangzhehui@pku.edu.cn}
\author{Jintian Zhu}
\address{Key Laboratory of Pure and Applied Mathematics, School of Mathematical Sciences, Peking University, Beijing 100871, China}
\email{zhujt@pku.edu.cn}

\begin{abstract}
For $n\geq2,$ we obtain Liouville type theorems for minimal surface equations in half space $\R^n_+$ with affine Dirichlet boundary value or constant Neumann boundary value.
\end{abstract}

\maketitle

\section{Introduction}

Liouville type theorems for several kinds of nonlinear elliptic equations in half space have already been extensively studied. For the semilinear elliptic equation $-\Delta u=|u|^{p-2}u$ in $\R_+^n$ with zero-Dirichlet boundary condition when $n\geq3$ and $2<p<2n/(n-2),$ Gidas and Spruck \cite{GiSp} proved that $u=0$ is the unique non-negative solution. For the real Monge-Amp\`{e}re equation, it is well known in Savin \cite{Sa} and  Mooney \cite{Mo} that any convex viscosity solution of $\det\nabla^2u=1$ in $\R_+^n$ with quadratic boundary condition must be a quadratic polynomial if $u=O(|x|^2)$ as $|x|\to\infty$. For minimal surface system prescribed with an affine Dirichlet boundary condition, Ding, Jost and Xin proved in \cite{DiJoXi} that any $C^2(\R_+^n)\cap C^{1,\alpha}(\overline{\R^n_+})$ solution with small singular values and uniformly bounded gradient must be an affine function, whose one-codimensional case indicates the validity of a Liouville type theorem for minimal graph over half space $\R^n_+$. Indeed, we can establish the following Liouville type theorem:

\begin{theorem}\label{ld}
Let $n\geq 2$ be an integer and $u\in C^2(\R_+^n)\cap C(\overline{\R_+^n})$ be a solution of
\begin{align}
\div\(\frac{\nabla u}{\sqrt{1+|\nabla u|^2}}\)&=0\,\,\, \hbox{in $\R_+^n$,}\label{mse}\\
u&=l\,\,\, \hbox{on $\partial\R_+^n,$}\label{diri}
\end{align}
where $l:\R^n\to\R$ is an affine function.
Assume that $u:\overline{\R_+^n}\to\R$ has at most a linear growth, which means there exists a constant $K>0$ such that
\begin{equation}\label{lingro}
|u(x)|\leq K(|x|+1)\,\,\, \hbox{for any $x\in\overline{\R_+^n}$.}
\end{equation}
Then $u$ is an affine function.
\end{theorem}

\begin{remark}
With the boundary condition \eqref{diri}, we point out that $u$ is smooth up to the boundary, which follows from an approximation procedure and the $C^{1,\alpha}$-estimates for quasilinear elliptic equations. For the convenience of the reader, we provide the details in Appendix \ref{B}.
\end{remark}

For entire minimal graph, Simons \cite{Sim} proved that any minimal graph over $\R^n$ must be a hyperplane for $2\leq n\leq7$. It is of particular interest to know whether the assumption \eqref{lingro} is necessary for above theorem. In two-dimensional case, it follows from the Schwarz reflection principle and Choi-Schoen curvature estimate \cite{CiSc} for minimal surfaces in $\R^3$ that Theorem \ref{ld} is true without the linear growth condition. To the best of our knowledge, the answer is still not clear in higher dimensional cases.
With the idea of reflection, it is fairly easy to prove Theorem \ref{ld} in case $l\equiv0$. To see this, we perform a Schwarz reflection for $u$ to obtain a new function $\tilde u$. Then, $\tilde u$ is an entire solution of minimal surface equation which has at most a linear growth. Theorem \ref{ld} then follows from the Liouville theorem for entire minimal graph.  When $l$ is a general affine function, the Schwarz reflection may not lead to an entire minimal graph, which appears to be a difficulty for Theorem \ref{ld}.

Different from the linear growth condition, we point out that the affine boundary value \eqref{diri} can not be removed. Otherwise, one may consider the function $$f(x)=\int_{1}^{|x|}{dt\over\sqrt{t^2-1}}$$
over the half plane
$$
P_+=\{(x_1,x_2)\in \R^2;\ x_2>2\}.
$$
Through direct calculations, it is quick to see that the function $f$ is a smooth solution of the minimal surface equation and that $|\nabla f|$ is uniformly bounded, but $f$ is not affine. From this point of view, it is interesting to know whether Liouville type theorem will be valid for Neumann boundary condition. The answer is definitely positive. In fact, we prove the following Liouville type theorem for constant Neumann boundary condition.

\begin{theorem}\label{ln}
Let $n\geq 2$ be an integer and $u\in C^2(\R_+^n)\cap C^1(\overline{\R_+^n})$ be a solution of \eqref{mse} with Neumann boundary condition
\begin{equation}
\partial_{x_n}u=\tau\ \ \ \hbox{on $\partial\R_+^n,$}\label{neu}
\end{equation}
where $\tau\in\R$ is a constant. If $u$ satisfies \eqref{lingro}, then $u$ is an affine function.
\end{theorem}

\begin{remark}
With the boundary condition \eqref{neu}, we note that $u$ is smooth up to the boundary. For more details, we refer the reader to \cite[Theorem 4.5]{Li} and \cite[Theorems 6.30-6.31]{GilTr}.
\end{remark}

As in the Dirichlet case, it is also not clear in the Neumann case whether the linear growth condition is necessary when $n\leq 7$. However, we notice $u$ is affine provided it is a solution of \eqref{mse} satisfying \eqref{diri} and \eqref{neu}, which is a direct conclusion from unique continuation property due to \cite{GaLi}.

In the following, we sketch the proof for our main theorems. For the Dirichlet case, the key ingredient is to obtain a uniform gradient estimate for the solution with affine boundary value and linear growth condition. For this purpose, we establish a boundary gradient estimate first, then the classical Bernstein technique due to \cite{Wa} yields the desired estimate. It turns out that the scaling invariance of the minimal surface equation and comparison principle make the linear growth come into play for boundary gradient estimate. To be explicit, we fix a weakly mean convex domain with some flat portion $T$ of $\R^n_+$ on its boundary. After imposing a particular smooth boundary value that coincides with the affine one on $T$, we obtain a smooth solution to the minimal surface equation as a comparison function. Compared with the rescaled solution $u_R(\cdot)=R^{-1}u(R\cdot)$, we obtain a uniformly bounded boundary gradient estimate. Then with the uniform gradient estimate derived from Bernstein method, we have a H\"older estimate for $\partial_{x_n}u$ by Krylov \cite{Kr}, which deduces that $\partial_{x_n}u$ is a constant by a scaling argument. At this stage, Theorem \ref{ld} follows easily from unique continuation \cite{GaLi} or Theorem \ref{ln}.

The proof follows a similar line for Neumann case. For gradient estimate, we apply the Bernstein method as usual but with a modified function to avoid its maximum appearing on the boundary, where the idea to push the maximum point inside is inspired from the work in Ma and Xu \cite{MaXu}. With the uniformly bounded gradient, we obtain the H\"older estimate for $\partial_{x_n}u$, which yields $\partial_{x_n}u$ to be a constant using the scaling argument. In this case, we can express $u$ to be a sum of $\tau x_n$ and an entire solution of minimal surface equation in $\R^{n-1},$ hence $u$ is affine.

The rest of this paper will be organized as follows. In section 2, we present details for gradient estimates in both Dirichlet and Neumann boundary condition. In section 3, we prove our main theorems.

\section{Gradient Estimate}

Throughout this paper, following notation will be used frequently.
\begin{itemize}
\item[(i)] For $i,j\in\{1,\ldots,n\},$ the Kronecker symbol $\delta_{ij}$ is given by 
\begin{equation*}
\delta_{ij}=\begin{cases}
1&\text{if $i=j,$}\\
0&\text{if $i\neq j.$}\\
\end{cases}
\end{equation*}
\item[(ii)] We particularly distinguish the $n$-th component and write $$x=(x',x_n)\ \ \ \hbox{for any $x\in\R^{n},$}$$
where $x'\in\R^{n-1}$ and $x_n\in\R.$
\item[(iii)] For $r>0,$\ $B_r(x_0)$ is the open ball of radius $r$ and center $x_0$ in $\R^n$, and $$B_r^+(x_0)=B_r(x_0)\cap\R_+^n.$$
If $x_0=0,$ we use $B_r$ to briefly represent $B_r(x_0)$ and $$\S_r=\overline{B_r}\cap\{x\in\R^n;\ x_n=0\},\ B_r^+=B_r\cap\R_+^n.$$
\item[(iv)] $C$ denotes a positive universal constant depending only on $n$ and $K$, whose meaning may be different from line to line.
\item[(v)]In section \ref{2.2}, we will use subscripts to write derivatives as$$(\cdot)_i=\partial_{x_i}(\cdot),\ (\cdot)_{ij}=\partial_{x_ix_j}(\cdot),\ (\cdot)_{ijk}=\partial_{x_ix_jx_k}(\cdot)$$
for brevity, whose meanings will be different from those subscripts of coefficients $a_{ij}.$
\end{itemize}

\subsection{Gradient Estimate for Dirichlet Problem}
In this subsection, we present the proof of global gradient estimate for Dirichlet case. We begin with the boundary gradient estimate as following.

\begin{lemma}\label{bgd}
Assume $u\in C^2(\overline{\R_+^n})$ is any solution of \eqref{mse}-\eqref{diri} satisfying \eqref{lingro}. Then
\begin{equation}\label{pd1}
\sup_{x\in\partial\R_+^n}|\nabla u(x)|<C
\end{equation}
for some universal constant $C>0$ (independent of $u$).
\end{lemma}
\begin{proof}
Without loss of generality, we may assume that $l(0)=0.$ It suffices to show that
\begin{equation}\label{pd2}
\sup_{x\in\Sigma_1}|\nabla u(x)|<C
\end{equation}
for some universal constant $C>0$.
To be explicit, we set $$u_R(x)={1\over R}u(Rx)\ \ \ \hbox{for $R>1$ and $x\in\overline{\R_+^n}$}.$$
It is clear that $u_R\in C^2(\overline{\R_+^n})$ solves \eqref{mse}-\eqref{diri}. For $R>1,$ it follows from \eqref{lingro} that $$
|u_R(x)|\leq K|x|+{K\over R}\leq K(|x|+1)\ \ \ \hbox{for every $x\in\R_+^n.$}
$$
Thanks to \eqref{pd2}, we have $$\sup_{x\in\Sigma_1}|\nabla u(Rx)|=\sup_{x\in\Sigma_1}|\nabla u_R(x)|<C\ \ \ \hbox{for every $R>1$,}$$
which implies \eqref{pd1}.

Now we turn to the proof of \eqref{pd2}. Let $\Omega\subset\R^n$ be the convex domain constructed in the Appendix \ref{A}, then the convexity of $\Omega$ implies that the boundary mean curvature $H_{\partial\Omega}$ is non-negative. Choose a smooth function $\rho:[0,+\infty)\to\R$ such that
\begin{align*}
&\rho(t)\equiv0\ \ \ \hbox{for every $t\in[0,1],$}\\
&\rho(t)\in[0,1]\ \ \ \hbox{for every $t\in[1,2],$}\\
&\rho(t)\equiv1\ \ \ \hbox{for every $t\in[2,\infty).$}
\end{align*}
For $x\in\overline\Omega$, set
\begin{equation*}
\phi(x)=6K\rho(|x|)+(1-\rho(|x|))l(x),
\end{equation*}
then $\phi\in C^3(\overline\Omega).$ By the construction of $\Omega,$ we have
\begin{equation}\label{pd5}
\phi(x)=6K\ \ \ \hbox{if $x\in\partial\Omega\cap\R_+^n;$} \,\,\,\,\,\,\,\,\phi(x)=l(x)\ \ \ \hbox{if $x\in\Sigma_1.$}
\end{equation}
Let $v\in C^2(\overline\Omega)$ be a solution to following Dirichlet problem
$$
\div\left({\nabla v\over\sqrt{1+|\nabla v|^2}}\right)=0\ \hbox{in $\Omega,$}\ \ \ v=\phi\ \hbox{on $\partial\Omega,$}
$$
whose existence is given by \cite{JeSe} or \cite[Theorem 16.10]{GilTr}. Notice that
\begin{equation*}
u(x)\leq K(|x|+1)\leq6K\ \ \ \hbox{for every $x\in\overline\Omega,$}
\end{equation*}
with \eqref{diri}, \eqref{pd5} and the maximum principle, we know $u\leq v$ in $\Omega.$ By (\ref{pd5}), we have $$
\partial_{x_n}u(x)\leq\partial_{x_n}v(x)\leq|\nabla v|_{L^\infty(\Omega)}\ \ \ \hbox{for every $x\in\Sigma_1$}.$$
A similar fashion gives $$\partial_{x_n}u(x)\geq-|\nabla v|_{L^\infty(\Omega)}\ \ \ \hbox{for every $x\in\Sigma_1.$}$$ Therefore, we get a uniform bound for $|\nabla u|$ on $\Sigma_1,$ which yields (\ref{pd2}).
\end{proof}
Using the classical Bernstein technique, we have
\begin{lemma}\label{ggd}
Let $n\geq2$ be an integer and $u\in C^2(\overline{\R_+^n})$ be a solution of \eqref{mse}. Assume $u$ satisfies \eqref{lingro} and \eqref{pd1}. Then $|\nabla u|\in L^\infty(\R_+^n).$
\end{lemma}
\begin{proof}
Following the calculation in \cite{Wa}, the only difficulty in our case is that the maximum point of the auxiliary function may locate on $\partial\R_+^n$. However, this can be overcome by \eqref{bgd}.
\end{proof}

\subsection{Gradient Estimate for Neumann Problem}\label{2.2}
In this subsection, we apply the Bernstein method to derive the global gradient estimate for Neumann problem.

\begin{lemma}\label{ggn}
If $u\in C^2(\overline{\R_+^n})$ satisfies \eqref{mse}, \eqref{lingro} and \eqref{neu}, then $|\nabla u|\in L^\infty(\R_+^n).$
\end{lemma}

\begin{proof}
According to the Lemma \ref{ggd}, we suffice to provide an upper bound for $|\nabla u|$ on $\partial\R_+^n$. 
To this end, it is enough to show that 
\begin{align}|\nabla u(0)|\leq C\label{gradu0}
\end{align} for some universal constant $C>0.$ To be explicit, for each $x_0\in\partial\R_+^n$ and $R>1+|x_0|$, we set $$
u_R(x)=\frac{1}{R}u(Rx+x_0)\ \ \ \hbox{for $x\in\overline{\R_+^n}.$}$$
Note that $u_R$ still satisfies \eqref{mse}, \eqref{lingro} and \eqref{neu}. Since $\nabla u(x_0)=\nabla u_R(0),$ we have $|\nabla u(x_0)|\leq C.$ 

We now prove \eqref{gradu0}. In what follows, we assume $\tau\geq0$ and $|\nabla u(0)|\geq(10+n+\tau)^{10}$. Restricting $u$ on $\overline{B_2^+(y_0)}$ for $y_0=(0,\ldots,0,1)\in\R^n,$ we may assume$$0\leq u\leq M=8K\ \hbox{in $B_2^+(y_0);$}$$ otherwise, we consider $u-\inf\limits_{B_2^+(y_0)}u$ instead. Set
\begin{align*}
&\eta(x)=\(1-\frac{|x-y_0|^2}{4}\)^2\ \ \ \hbox{for $x\in\overline{\R_+^n}$},\\
&\gamma(t)=1+\frac{t}{M}\ \ \ \hbox{for $0\leq t\leq M,$}\\
&w(x)=u(x)-\tau x_n\ \ \ \hbox{for $x\in\overline{\R_+^n}$},\\
&\Omega=\left\{x\in\overline{B_2^+(y_0)};\ |\nabla w|\geq(10+n+\tau)^{10}\right\},
\end{align*}
and we define $$\Phi(x)=\eta(x)\gamma(u(x))\log|\nabla w|^2\ \ \ \hbox{for $x\in\Omega.$}$$
Then $\Phi$ attains its maximum at some $y_1\in\Omega\setminus\partial B_2(y_0).$ The rest of the proof will be divided into three cases.

{\sl Case 1.} $y_1\in\partial\R_+^n.$ In this case,
$$(\log\Phi)_n(y_1)=\frac{4}{4-|y_1-y_0|^2}+\frac{\tau}{M+u(y_1)}>0,$$ 
which is a contradiction.

{\sl Case 2.} $y_1\in\partial\Omega.$ In this case, we also have $\log|\nabla w(0)|\leq C\log(10+n+\tau).$

{\sl Case 3.} $y_1$ is an interior point of $\Omega.$ Then we have $\nabla(\log\Phi)=0$ and $\nabla^2(\log\Phi)$ is negative definite at $y_1.$ In some neighborhood of $y_1$, the minimal surface equation can be written as 
\begin{align}
\sum_{i,j=1}^na_{ij}(\nabla u)u_{ij}=0,\label{ndmse}
\end{align}
where each coefficient $a_{ij}$ is given by $$a_{ij}(p)=\delta_{ij}-\frac{p_ip_j}{1+|p|^2}\ \ \ \hbox{for $p=(p_1,\ldots,p_n)\in\R^n.$}$$ 

In order to simply our calculation, we choose a suitable coordinate such that 
\begin{align}
&u_1(y_1)=|\nabla u(y_1)|>0,\,\,\, u_i(y_1)=0\ \ \ \hbox{for $2\leq i\leq n,$}\label{ui}\\
&u_{ij}(y_1)=\lambda_i\delta_{ij}\ \ \ \hbox{for $2\leq i,j\leq n,$}\label{uij}
\end{align}
then $w_1(y_1)\geq u_1(y_1)-\tau>0$ and $$\sum_{i=2}^n|w_i(y_1)|^2\leq\tau^2.$$
Under the new coordinate, $u$ still satisfies the minimal surface equation \eqref{ndmse}, it follows from \eqref{ui} and \eqref{uij} that
\begin{align}
&a_{11}(\nabla u(y_1))=\frac{1}{1+u_1^2},\ a_{ii}=1\ (i\neq1),\ a_{ij}(\nabla u(y_1))=0\ (i\neq j),\label{aij}\\
&\partial_{p_1}a_{11}(\nabla u(y_1))=-\frac{2u_1}{(1+u_1^2)^2},\ \partial_{p_i}a_{11}(\nabla u(y_1))=0\ (i\neq1),\label{pa11}\\ 
&\partial_{p_i}a_{1i}(\nabla u(y_1))=-\frac{u_1}{1+u_1^2}\ (i\neq1),\ \partial_{p_j}a_{1i}=0\ (j\neq1,i\neq j),\\
&\partial_{p_k}a_{ij}(\nabla u(y_1))=0\ (i\neq1,j\neq1).\label{paij}
\end{align}

In the following, we work at the point $y_1$ to evaluate $$\sum\limits_{i,j=1}^na_{ij}(\nabla u)(\log\Phi)_{ij}.$$
Through a simple differentiation, there holds 
\begin{align}\label{Phi}
(\log\Phi)_i=\frac{\eta_i}{\eta}+\frac{\gamma'} {\gamma}u_i+\frac{(|\nabla w|^2)_i}{|\nabla w|^2\log|\nabla w|^2}=0.
\end{align}
Since
\begin{align*}
(\log\Phi)_{ij}=\frac{(|\nabla w|^2)_{ij}}{|\nabla w|^2\log|\nabla w|^2}-&(1+\log|\nabla w|^2)\frac{(|\nabla w|^2)_i(|\nabla w|^2)_j}{|\nabla w|^4\log^2|\nabla w|^2}\\
 &+\frac{\eta_{ij}}{\eta}-\frac{\eta_i\eta_j}{\eta^2}+\left(\frac{\gamma''}{\gamma}- \frac{(\gamma')^2}{\gamma^2} \right)u_iu_j+\frac{\gamma'}{\gamma}u_{ij},
 \end{align*}
 by \eqref{Phi} and a direct substitution, we obtain 
 \begin{align}
 \begin{split}
 (\log\Phi)_{ij}=\frac{(|\nabla w|^2)_{ij}}{|\nabla w|^2\log|\nabla w|^2}-&(1+\log|\nabla w|^2)\left(\frac{\eta_i}{\eta}+\frac{\gamma'}{\gamma}u_i\right) \left(\frac{\eta_j}{\eta}+\frac{\gamma'}{\gamma}u_j\right)\\
 &+ \frac{\eta_{ij}}{\eta}-\frac{\eta_i\eta_j}{\eta^2}+\left(\frac{\gamma''}{\gamma}- \frac{(\gamma')^2}{\gamma^2} \right)u_iu_j+\frac{\gamma'}{\gamma}u_{ij}.
\end{split}
\label{logPhi}
\end{align}
By combining \eqref{aij} and \eqref{logPhi}, there holds
\begin{align*}
\sum_{i,j=1}^na_{ij}(\nabla u)(\log\Phi)_{ij}=&\frac{1}{|\nabla w|^2\log|\nabla w|^2}\sum_{i,j=1}^na_{ij}(\nabla u)(|\nabla w|^2)_{ij}\\
&-(1+\log|\nabla w|^2)\sum_{i=1}^n\left\{\frac{1}{1+u_1^2}\left(\frac{\eta_1}{\eta}+\frac{\gamma'}{\gamma}u_1\right)^2+\sum_{i=2}^n\frac{\eta_i^2}{\eta^2}\right\}\\
&+\left\{\frac{1}{1+u_1^2}\left(\frac{\eta_{11}}{\eta}-\frac{\eta_1^2}{\eta^2}-\frac{u_1^2}{M^2\gamma^2}\right)+\sum_{i=2}^n\left(\frac{\eta_{ii}}{\eta}-\frac{\eta_i^2}{\eta^2}\right)\right\}\\
=:&I_1+I_2+I_3,
\end{align*}
then it is clear that $I_2$ is a negative term, we will use $I_1$ to bound $I_2$. We also note that $$|\nabla^2\eta|+\frac{|\nabla\eta|^2}{\eta}\leq C\ \ \ \hbox{in $B_2(y_0)$},$$ 
thus \begin{align}
I_3\geq-\(\frac{C}{\eta}+\frac{1}{M^2}\).&\label{I3}
\end{align}

A straightforward calculation yields that 
\begin{align}
(|\nabla w|^2)_{i}&=2\sum_{k=1}^nu_{ik}w_k,\label{gw}\\
(|\nabla w|^2)_{ij}&=2\sum_{k=1}^n(u_{ik}u_{jk}+u_{ijk}w_k),\nonumber
\end{align}
thus we obtain from \eqref{uij} and \eqref{aij} that
\begin{align}
\sum_{i,j=1}^na_{ij}(\nabla u)(|\nabla w|^2)_{ij}=&2\sum_{i,k=1}^na_{ii}(\nabla u)u_{ik}^2+2\sum_{i,j,k=1}^na_{ij}(\nabla u)u_{ijk}w_k\nonumber\\
=&\frac{2}{1+u_1^2}\sum_{k=1}^nu_{1k}^2+2\sum_{i=2}^nu_{1i}^2+2\sum_{i=2}^nu_{ii}^2\nonumber\\
&+2\sum_{k=1}^n\sum_{i,j=1}^na_{ij}(\nabla u)u_{ijk}w_k.\label{uijk}
\end{align}
In order to eliminate the third derivatives of $u$ in \eqref{uijk}, we differentiate the minimal surface equation \eqref{ndmse} and get $$\sum_{i,j=1}^na_{ij}(\nabla u)u_{ijk}+\sum_{i,j,l=1}^n\partial_{p_l}a_{ij}(\nabla u)u_{ij}u_{kl}=0,$$ thus we obtain from \eqref{pa11}-\eqref{paij} that 
\begin{align*}
\sum_{i,j=1}^na_{ij}(\nabla u)u_{ijk}=\frac{2u_1u_{11}u_{1k}}{(1+u_1^2)^2}+\frac{2u_1}{1+u_1^2}\sum_{j=2}^nu_{1j}u_{jk}\ \ \ \hbox{for $k=1,\ldots,n$}.\end{align*}
By a simple substition, we get  
$$(|\nabla w|^2\log|\nabla w|^2)I_1=\sum_{i,j=1}^na_{ij}(\nabla u)(|\nabla w|^2)_{ij}=:J_1+J_2,$$
where 
\begin{align*}
J_1&=\frac{(2+2u_1^2+4u_1w_1)u_{11}^2}{1+u_1^2}+\frac{4+2u_1^2+4u_1w_1}{1+u_1^2}\sum_{i=2}^nu_{1i}^2+2\sum_{i=2}^n\lambda_i^2,\\
J_2&=\frac{4u_1u_{11}}{(1+u_1^2)^2}\sum_{k=2}^nu_{1k}w_k+\frac{4u_1}{1+u_1^2}\sum_{i=2}^n\lambda_iu_{1i}w_i.
\end{align*}
We point out that $-|J_2|$ can be bounded by $J_1,$ to see this, we apply the Cauchy inequality to get\begin{align*}
\frac{4u_1}{(1+u_1^2)^2}\sum_{k=2}^n|u_{11}u_{1k}w_k|&\geq-\frac{4\tau u_1}{(1+u_1^2)^2}\sum_{k=2}^n|u_{11}u_{1k}|\\
&\geq-\frac{2\tau u_1}{(1+u_1^2)^2}\((n-1)u_{11}^2+\sum_{k=2}^nu_{1k}^2\),\\
\frac{4u_1}{1+u_1^2}\sum_{i=2}^n|\lambda_iu_{1i}w_i|&\geq-\frac{-4\tau u_1}{1+u_1^2}\sum_{i=2}^n|\lambda_iu_{1i}|\\
&\geq-\frac{2\tau u_1}{1+u_1^2}\(\sum_{i=2}^n\lambda_i^2+\sum_{i=2}^nu_{1i}^2\).
\end{align*}
Since $u_1\geq(10+n+\tau)^{10}$ and $$\frac{w_1}{u_1}\geq1-\frac{\tau}{u_1}\geq1-\frac{\tau}{(10+n+\tau)^{10}},$$ 
we have $$\frac{2+2u_1^2+4u_1w_1}{1+u_1^2}\geq\frac{11}{2},\ \hbox{and}\ \frac{2(n-1)\tau u_1}{(1+u_1^2)^2}+\frac{2\tau u_1}{1+u_1^2}\leq\frac{1}{5},$$ which imply
\begin{align}\(|\nabla w|^2\log|\nabla w|^2\)I_1\geq J_1-|J_2|\geq\frac{49}{10}\sum_{i=2}^nu_{1i}^2+\frac{49}{10}u_{11}^2+\frac{9}{5}\sum_{i=2}^n\lambda_i^2.&\label{I1}
\end{align}
Now we start to deal with the $\sum\limits_{i=2}^nu_{1i}^2$ and $u_{11}^2.$ By taking $i\geq2$ in \eqref{Phi}, we obtain from \eqref{uij} and \eqref{gw} that 
\begin{align}
u_{1i}&=-\frac{\lambda_iw_i}{w_1}-\frac{\eta_i|\nabla w|^2\log|\nabla w|^2}{2\eta w_1}.\label{u1i}
\end{align}
By taking $i=1$ in \eqref{Phi} and using \eqref{u1i}, we have \begin{align*}
u_{11}&=-\sum_{j=2}^n\frac{u_{1j}w_j}{w_1}-\frac{1}{2w_1}\left(\frac{\eta_1}{\eta}+\frac{\gamma'u_1}{\gamma}\right)|\nabla w|^2\log|\nabla w|^2\nonumber\\
&=\sum_{j=2}^n\frac{\lambda_jw_j^2}{w_1^2}+\frac{1}{2w_1}\left(\sum_{j=2}^n\frac{\eta_jw_j}{\eta w_1}-\frac{\eta_1}{\eta}-\frac{\gamma'u_1}{\gamma}\right)|\nabla w|^2\log|\nabla w|^2.
\end{align*}
Hence, for $\epsilon>w_1^{-4}$ to be determined, we have 
\begin{align*}
\sum_{i=2}^nu_{1i}^2&=\sum_{i=2}^n\left(\frac{\lambda_iw_i}{w_1}+\frac{\eta_i|\nabla w|^2\log|\nabla w|^2}{2\eta w_1}\right)^2\\
&\geq\frac{1}{4w_1^2}\sum_{i=2}^n\frac{\eta_i^2|\nabla w|^4\log^2|\nabla w|^2}{\eta^2}+\sum_{i=2}^n\frac{\lambda_i\eta_iw_i|\nabla w|^2\log|\nabla w|^2}{\eta w_1^2}\\
&\geq-\epsilon\tau^2\sum_{i=2}^n\lambda_i^2+\(\frac{1}{4w_1^2}-\frac{1}{4\epsilon w_1^4}\)\sum_{i=2}^n\frac{\eta_i^2|\nabla w|^4\log^2|\nabla w|^2}{\eta^2},\\
u_{11}^2\geq&\(\sum_{j=2}^n\frac{\lambda_jw_j^2}{w_1^3}\)\cdot\left(\sum_{j=2}^n\frac{\eta_jw_j}{\eta w_1}-\frac{\eta_1}{\eta}-\frac{\gamma'u_1}{\gamma}\right)|\nabla w|^2\log|\nabla w|^2\\
&+\frac{1}{4w_1^2}\left(\sum_{j=2}^n\frac{\eta_jw_j}{\eta w_1}-\frac{\eta_1}{\eta}-\frac{\gamma'u_1}{\gamma}\right)^2|\nabla w|^4\log^2|\nabla w|^2\\
\geq&-\epsilon\tau^4\sum_{j=2}^n\lambda_j^2+\(\frac{1}{4w_1^2}-\frac{1}{4\epsilon w_1^6}\)\left(\sum_{j=2}^n\frac{\eta_jw_j}{\eta w_1}-\frac{\eta_1}{\eta}-\frac{\gamma'u_1}{\gamma}\right)^2|\nabla w|^4\log^2|\nabla w|^2\\
\geq&-\epsilon\tau^4\sum_{j=2}^n\lambda_j^2+\(\frac{1}{4w_1^2}-\frac{1}{4\epsilon w_1^6}\)|\nabla w|^4\log^2|\nabla w|^2\\
&\cdot\(-\epsilon\tau^2\sum_{j=2}^n\frac{\eta_j^2}{\eta^2}+\(1-\frac{1}{\epsilon w_1^2}\)\(\frac{\eta_1}{\eta}+\frac{\gamma'u_1}{\gamma}\)^2\).
\end{align*}
Taking $\epsilon=4w_1^{-2}>2w_1^{-4},$ then $100\epsilon<(1+\tau^2+\tau^4)^{-1}.$ Thus we obtain from \eqref{I1} that 
\begin{align*}
I_1\geq&\frac{11}{8w_1^2}\(1-\epsilon\tau^2-\frac{1}{\epsilon w_1^2}\)\sum_{i=2}^n\frac{\eta_i^2|\nabla w|^2\log|\nabla w|^2}{\eta^2}\\
&+\frac{11}{8w_1^2}\(1-\frac{1}{\epsilon w_1^6}\)\(1-\frac{1}{\epsilon w_1^2}\)\(\frac{\eta_1}{\eta}+\frac{\gamma'u_1}{\gamma}\)^2|\nabla w|^2\log|\nabla w|^2\\
\geq&\frac{407\log|\nabla w|^2}{400}\sum_{i=2}^n\frac{\eta_i^2}{\eta^2}+\frac{99}{128}\(\frac{\eta_1}{\eta}+\frac{\gamma'u_1}{\gamma}\)^2\log|\nabla w|^2.
\end{align*}
Therefore,
\begin{equation}
\begin{split}
I_1+I_2&\geq\left(\frac{99\log|\nabla w|^2}{128}-\frac{1+\log|\nabla w|^2}{1+u_1^2}\right)\(\frac{\eta_1}{\eta}+\frac{\gamma'u_1}{\gamma}\)^2\\
&\geq\frac{\log|\nabla w|^2}{2}\(\frac{\eta_1}{\eta}+\frac{\gamma'u_1}{\gamma}\)^2.
\end{split}
\label{I1I2}
\end{equation}
We note that $$\sum_{i,j=1}^na_{ij}(\nabla u)(\log\Phi)_{ij}=I_1+I_2+I_3\leq0\ \ \ \hbox{at $y_1,$}$$ hence 
\eqref{I1I2} and \eqref{I3} combined give that $$
\frac{\log|\nabla w|^2}{2}\(\frac{\eta_1}{\eta}+\frac{\gamma'u_1}{\gamma}\)^2\leq\frac{C}{\eta}+\frac{1}{M^2}.$$
To end this proof, two cases will be treated in what follows. First, if
$$
\left|\frac{\eta_1}{\eta u_1}\right|\leq \frac{\gamma'}{2\gamma},
$$
then
$$
\log |\nabla w|^2\leq C\frac{\gamma^2}{(\gamma')^2}\left(\frac{1}{\eta}+\frac{1}{M^2}\right).
$$
It follows
$$
\log|\nabla w(0)|^2\leq C\eta\log|\nabla w|^2\leq CM^2
$$
Second, if
$$
\left|\frac{\eta_1}{\eta u_1}\right|> \frac{\gamma'}{2\gamma},
$$
then
\begin{align*}
\log|\nabla w(0)|&\leq C\Phi(y_1)\leq C\eta u_1\leq\frac{2\gamma|\eta_1|}{\gamma'}\leq CM.
\end{align*}
To sum up, we have
$$
|\nabla w(0)|\leq\exp\left\{C(M^2+M)\right\}.$$
This completes the proof.
\end{proof}

\section{Proof of Main Theorems}
\begin{proof}[Proof of Theorem \ref{ln}]
For any $R>0$ and $x\in\overline{\R_+^n},$ we set $$u_R(x)=\frac{1}{R}u(Rx)\ \hbox{and}\ v_R(x)={\partial_{x_n} u_R}(x),$$ then it follows from Lemma \ref{ggn} that $|\nabla u_R|\in L^\infty(\R_+^n).$  A basic calculation shows that $v_R\in  C(\overline{\R_+^n})\cap C^\infty(\R^n_+)$ satisfies following uniform elliptic equation with constant Dirichlet boundary value:
\begin{equation*}
\displaystyle\sum_{i, j=1}^n{\partial_{x_i}}\(b_{R, ij}{\partial_{x_j}v_R}\)=0\ \hbox{in $\R_+^n$,}\ \ \ 
v_R=\tau \ \hbox{on $\partial\R_+^n$,}
\end{equation*}
where $$b_{R,ij}={1\over\sqrt{1+|\nabla u_R|^2}}\(\delta_{ij}-{{\partial_{x_i}u_R}{\partial_{x_j}u_R}\over 1+|\nabla u_R|^2}\).$$
From \cite[Theorem 8.27, Theorem 8.29]{GilTr}, there exists $\alpha\in(0,1)$ such that
\begin{equation*}
\|v_R\|_{C^\alpha(\overline{B_1^+})}\leq C\|v_R\|_{L^2(B_2^+)}\leq C.
\end{equation*}
Therefore,
$$|v_R(x)-v_R(0)|\leq C|x|^\alpha \,\,\, \hbox{for any $x\in\overline{B_1^+} $},$$
which yields 
\begin{equation}\label{hol}
|{\partial_{x_n}u}(y)-{\partial_{x_n}u}(0)|\leq C{|y|^\alpha\over R^\alpha} \,\,\, \hbox{for any $y\in\overline{B_R^+} $}.
\end{equation}
Fixing $y$ and letting $R\rightarrow \infty$ in \eqref{hol}, we know ${\partial_{x_n}u}$ is a constant in $\overline{\R_+^n}.$ Hence $u(x', x_n)=\tilde{u}(x')+\tau x_n$, where $\tilde{u}$ is a smooth function on $\R^{n-1}$. Further $\tilde{u}$ is an entire solution of the minimal surface equation in $\R^{n-1}$, which means $\tilde{u}$ is an affine function in  $\R^{n-1}$ by Liouville theorem for entire minimal graph. The proof is finished.
\end{proof}

\begin{proof}[Proof of Theorem \ref{ld}]
Without loss of generality, we may assume that there exists $\beta=(\beta_1,\ldots,\beta_n)\in\R^n$ with $\beta_n=0$ such that $$l(x)=\langle\beta,x\rangle\ \ \ \hbox{for any $x\in\R_+^n.$}$$
Set $\bar u=u-l$, then $|\nabla\bar u|\in L^\infty(\R_+^n)$ and $\bar u\in C^2(\overline{\R_+^n})$ satisfies 
\begin{equation*}
\sum_{i,j=1}^n\bar a_{ij}\partial_{x_ix_j}\bar u=0\ \hbox{in $\R_+^n$,}\ \ \ \bar u=0\ \hbox{on $\partial\R_+^n$,}
\end{equation*}
where $$\bar a_{ij}=\delta_{ij}-{(\partial_{x_i}\bar u+\beta_i)(\partial_{x_j}\bar u+\beta_j)\over{1+|\nabla\bar u+\beta|^2}}.$$
By the H\"{o}lder estimate for the normal derivatives of solutions on the boundary due to Krylov \cite{Kr} (see also \cite[Theorem 1.2.16]{Ha}), we get $$\partial_{x_n}u(x',0)=\partial_{x_n}\bar u(x',0)\equiv c\ \ \ \hbox{on $\partial\R_+^n,$}$$ where $c\in\R$ is a constant. 
Then $u$ is an affine function by Theorem \ref{ln}.
\end{proof}

\begin{appendix}
\section{Bounded Convex Domain with $C^3$-boundary}\label{A}

In the following, we construct a bounded convex $C^3$-type domain, whose boundary contains $\Sigma_1$. 

For $t\in[0,1],$ we set $$\psi(t)={64\over35}t^{1/2}-2t^2+{8\over5}t^3-{3\over7}t^4.$$ Then, $\psi:[0,1]\rightarrow[0,1]$ is continuous and concave. Straightforward calculations show that
\begin{align*}
&\psi>0,\, \psi'>0,\, \psi''<0\,\, \hbox{and}\,\, \psi'''>0\, \hbox{in $(0,1),$}\\
&\lim_{t\to0+}\psi'(t)=+\infty,\, \lim_{t\to1-}\psi'(t)=\psi'(1)=0,\\
&\lim_{t\to0+}\psi''(t)=-\infty,\, \lim_{t\to1-}\psi''(t)=\psi''(1)=0,\\
&\lim_{t\to0+}\psi'''(t)=+\infty,\, \lim_{t\to1-}\psi'''(t)=\psi'''(1)=0.\\
\end{align*}
For every $h\in[0,2],$ we define
\begin{equation*}
\tilde\psi(h)=
\begin{cases}
\psi(h) & \text{if $h\in[0,1],$}\\
\psi(2-h) & \text{if $h\in[1,2].$}\\
\end{cases}
\end{equation*}
and $$\Omega=\left\{(x',x_n)\in\R_+^n;\, |x'|<2+\tilde\psi(x_n),\, 0<x_n<2\right\},$$
we then claim that:
\medskip

\noindent{\bf Claim.} $\Omega\subset\R^n$ is a convex bounded domain with $C^3$-boundary.

\begin{proof} It is easy to see that $\Omega\subset B_5^+$ is a $C^3$-type domain. Let $(x',x_n)$ and $(y',y_n)$ be two points in $\Omega.$ We note that $\tilde\psi:[0,2]\to[0,1]$ is a concave function. For any $t\in[0,1]$, we obtain from the concavity of $\tilde\psi$ that
\begin{align*}
|tx'+(1-t)y'|&\leq t|x'|+(1-t)|y'|\\
&<t(2+\tilde\psi(x_n))+(1-t)(2+\tilde\psi(y_n))\\
&\leq2+t\tilde\psi(x_n)+(1-t)\tilde\psi(y_n)\\
&\leq2+\tilde\psi(tx_n+(1-t)y_n).
\end{align*}
Hence, we have $$(tx'+(1-t)y',tx_n+(1-t)y_n)\in\Omega,$$
which implies that $\Omega\subset\R^n$ is convex.
\end{proof}

\medskip
\section{Global Regularity for Solutions}\label{B}
In this section, we show that solutions of \eqref{mse}-\eqref{diri} are smooth up to the boundary $\partial\R_+^n,$ which is an immediate corollary of following general theorem.

\begin{theorem}
Let $\Omega\subset\R^n$ be a bounded domain with $C^3$-boundary satisfying $H_{\partial\Omega}\geq0$ on $\partial\Omega,$ where $H_{\partial\Omega}$ is the mean curvature of $\partial\Omega$ corresponding to the inner unit normal vector to $\partial\Omega.$ Suppose $T$ is a smooth portion of $\partial\Omega.$ For $\gamma\in(0,1)$ and $\varphi\in C(\partial\Omega)\cap C^{2,\gamma}(T),$ and $u\in C^2(\Omega)\cap C(\overline\Omega)$ solves the minimal surface equation \eqref{ndmse} in $\Omega$ with the Dirichlet boundary condition $u=\varphi$ on $\partial\Omega.$ Then,\ $u\in C^{2,\gamma}(\Omega\cup T).$ Furthermore, if $T$ is a smooth portion of $\partial\Omega$ and $\varphi\in C(\partial\Omega)\cap C^\infty(T),$ then $u\in C^\infty(\Omega\cup T).$
\end{theorem}

\begin{proof}
For $x_0\in T,$ put $$d={\rm dist}(x_0,\partial\Omega\setminus T)>0.$$
We suffice to show $\nabla u$ is well defined on $\overline{B_{d/16}(x_0)}\cap\overline\Omega,$ and $\nabla u\in C^{\alpha}(\overline{B_{d/16}(x_0)}\cap\overline\Omega)$ for some $\alpha\in(0,\gamma),$ which will implies $a_{ij}(\nabla u)\in C^\alpha(\overline{B_{d/16}(x_0)}\cap\overline\Omega).$ Then, by $\varphi\in C^{2,\gamma}(T)$, extension and \cite[Theorem 1.3.7]{Ha}, we know $u\in C^{2,\alpha}(\Omega\cup T),$ which implies $a_{ij}(\nabla u)\in C^{\gamma}(\Omega\cup T).$ Again we get $u\in C^{2,\gamma}(\Omega\cup T).$ As to the case that $T$ is smooth and $\varphi\in C^\infty(T),$ based on the \cite[Theorem 1.3.10]{Ha}, it will be finished by a bootstrap argument.

Let $\{\phi_k\}_{k=1}^\infty\subset C^{2,1/2}(\overline\Omega)$ be a sequence satisfying 
\begin{align*}
&\phi_k\to\varphi\ \ \ \hbox{in $C(\partial\Omega),$}\\
&\phi_k\to u\ \ \ \hbox{in $C^{2,1/2}(\overline{B_{d/2}(x_0)}\cap\partial\Omega).$}
\end{align*}
We assume $$|\phi_k|_{L^\infty(\partial\Omega)}\leq M_0,\ |\phi_k|_{C^{2,1/2}(\overline{B_{d/2}(x_0)}\cap\partial\Omega)}\leq M_0$$ for some large constant $M_0>0.$ With each boundary value $\phi_k,$ we solve the minimal surface equation in $\Omega$ to obtain a solution $u_k\in C^{2,1/2}(\overline\Omega)$. It follows from comparison principle that $u_k\to u$ in $C(\overline\Omega)$ and
\begin{equation*}
|u_k|_{L^\infty(\Omega)}\leq M_0\ \ \ \hbox{for all $k\in\N_+.$}
\end{equation*}

{\sl Step 1.} We estimate the $L^\infty$-norm of $|\nabla u_k|$ near $T$. To this end, we choose a cut-off function $\rho\in C_0^\infty(B_{d/2}(x_0))$ such that 
\begin{align*}
&\rho\equiv1\ \ \ \hbox{in $B_{d/4}(x_0),$}\\
&0\leq\rho\leq1\ \ \ \hbox{in $B_{d/2}(x_0),$}\\
&|\nabla\rho|\leq8/d\ \ \ \hbox{in $B_{d/2}(x_0).$}
\end{align*}
For each $k\in\N_+,$ we set $\psi_k=(1-\rho)\phi_k+\rho M_0,$ then $\phi_k\leq\psi_k$ on $\partial\Omega.$ As in the proof of Lemma \ref{bgd}, we construct the comparison function $v_k$ which is a solution of the minimal surface equation in $\Omega$ with the boundary value $\psi_k$ on $\partial\Omega.$ Hence, we obtain from Lemma \ref{bgd} that $$\partial_{x_n}u_k(x_0)\leq\partial_{x_n}v_k(x_0)\leq C(d,M_0),$$ thus we get $$|\nabla u_k|_{L^\infty(T\cap B_{d/4}(x_0))}\leq C(d,M_0),$$ where $C(d,M_0)$ are positive quantities depending only on $d$ and $M_0.$ Then, we can proceed similarly as in the proof of Lemma \ref{ggd} and apply the Bernstein technique to obtain $$|\nabla u_k|_{L^\infty(\Omega\cap B_{d/8}(x_0))}\leq C(d,M_0).$$

{\sl Step 2.}  We provide an upper bound for $C^{1,\alpha}$-norms of $\{u_k\}$ near $x_0$. Since the proof is almost similar to the proof of \cite[Theorem 2.5.1]{Ha}, we sketch the procedure in the following.

{\sl Step 2.1.} Flattening the boundary $\partial\Omega$ near $x_0$, then applying \cite[Theorem 1.2.16]{Ha} to $u_k-\phi_k$ to obtain the boundary H\"{o}lder estimate for normal derivative of $u_k-\phi_k.$ Notice that the estimate is done near $x_0$ and $$|\phi_k|_{C^2(\overline\Omega\cap\overline{B_{d/8}(x_0)})}\leq C(M_0),$$ it follows from the first step of proof of \cite[Theorem 2.5.1]{Ha} that 
\begin{equation}\label{B1}
[\nabla u]_{C^\beta(\overline T\cap\overline{B_{d/8}(x_0)})}\leq C(n,d,M_0,\Omega)
\end{equation}
for some $\beta=\beta(n,d,M_0,\Omega)\in(0,1).$

{\sl Step 2.2.} Using \eqref{B1} to prove that $$|\nabla u_k(x)-\nabla u_k(y)|\leq C(n,d,\beta,M_0,\Omega)|x-y|^{\beta/(1+\beta)}$$ for all $x\in\overline\Omega\cap\overline{B_{d/8}(x_0)}$ and $y\in \overline T\cap\overline{B_{d/16}(x_0)}.$

{\sl Step 2.3.} Based on the Step 2.2, directly following the Step 3 of the proof of \cite[Theorem 2.5.1]{Ha} to obtain $$[\nabla u_k]_{C^{\beta'}(\overline\Omega\cap\overline{B_{d/16}(x_0)})}\leq C(n,d,\beta,M_0,\Omega)$$ for some $\beta'\in(0,\beta).$

By the Arzel\`a-Ascoli theorem, we deduce that $\nabla u$ is well defined on $\overline\Omega\cap\overline{B_{d/16}(x_0)},$ and $\nabla u\in C^{\alpha}(\overline\Omega\cap\overline{B_{d/16}(x_0)})$ for any $\alpha\in(0,\beta').$
\end{proof}

\end{appendix}
\bigskip
\noindent
{\bf Acknowledgments.}
G. Jiang and Z. Wang would like to express their gratitude to Professor Qing Han for constant encouragements. J. Zhu is partially supported by NSFC grants No. 11671015 and 11731001. Authors would also like to thank Mr. Zhisu Li and Mr. Yongjie Liu for helpful discussions.

\end{document}